\documentclass[12pt]{amsart}

\usepackage{fullpage}

\usepackage{epsfig}
\usepackage{amssymb}
\usepackage{hyperref}
\newtheorem{theorem}{Theorem}[section]
\newtheorem{prop}[theorem]{Proposition}
\newtheorem{lemma}[theorem]{Lemma}

\newtheorem{conjecture}[theorem]{Conjecture}

{\theoremstyle{remark}

}

{\theoremstyle{definition}
\newtheorem{definition}[theorem]{Definition}
}

\newcommand{\CC}{\mathcal{C}}
\newcommand{\Av}{\operatorname{Av}}
\newcommand{\Si}{\operatorname{Si}}

\usepackage{tikz}

\author{Alexander Burstein}
\address{Department of Mathematics, Howard University, Washington, DC 20059}
\email{aburstein@howard.edu}
\urladdr{http://www.alexanderburstein.org}

\author{Jay Pantone}
\thanks{The second author\text{'}s research was sponsored by the National Science Foundation under Grant Number DMS-1301692.}
\address{Department of Mathematics, University of Florida, Gainesville, FL 32611}
\email{jaypantone@ufl.edu}
\urladdr{http://www.jaypantone.com}

\title{Two Examples of Unbalanced Wilf-Equivalence}

\date{\today}

\begin{document}

\begin{abstract}
We prove that the set of patterns $\{1324,3416725\}$ is Wilf-equivalent to the pattern $1234$ and that the set of patterns $\{2143,3142,246135\}$ is Wilf-equivalent to the set of patterns $\{2413,3142\}$. These are the first known unbalanced Wilf-equivalences for classical patterns between finite sets of patterns.
\end{abstract}

\maketitle

A \emph{pattern} is an equivalence class of sequences under order-isomorphism. Two sequences $\pi_1$ and $\pi_2$ over totally ordered alphabets are \emph{order-isomorphic} if, for any pair of positions $i$ and $j$, we have $\pi_1(i)<\pi_1(j)$ if and only if $\pi_2(i)<\pi_2(j)$.  We identify a pattern with its canonical representative, in which the $k$th smallest letter is $k$. We say that a permutation $\pi$ \emph{contains} a pattern $\sigma$ if $\pi$ has a subsequence order-isomorphic to $\sigma$, otherwise we say that $\pi$ \emph{avoids} $\sigma$.

We denote the set of permutations of length $n$ by $S_n$, the set of permutations in $S_n$ avoiding pattern $\pi$ by $\Av_n(\pi)$, and the set of permutations in $S_n$ avoiding every pattern in a set $\Pi$ by $\Av_n(\Pi)$. We say that two (sets of) patterns $\pi'$ and $\pi''$ are \emph{Wilf-equivalent}, denoted $\pi'\sim\pi''$, if $|\Av_n(\pi')|=|\Av_n(\pi'')|$ for all $n\in\mathbb{N}$. We call a Wilf-equivalence \emph{unbalanced} if the two sets of patterns do not contain the same number of patterns of each length. We will sometimes talk about the \emph{type} of an unbalanced Wilf-equivalence, defined by the lengths of the patterns: for example, the two Wilf-equivalences proved in this paper, $1234 \sim \{1324,3416725\}$ and $\{2413,3142\} \sim \{2143,3142,246135\}$, have type $(4) \sim (4,7)$ and $(4,4) \sim (4,4,6)$, respectively.

In Section \ref{sec:4-47}, we prove that $\{1324,3416725\}\sim 1234$. In Section \ref{sec:44-446}, we prove that\linebreak $\{2143,3142,246135\}\sim\{2413,3142\}$. In Section \ref{sec:conj}, we conjecture a few other unbalanced Wilf-equivalences.

\section{A $(4)\sim(4,7)$ Wilf-Equivalence} \label{sec:4-47}

\begin{theorem} \label{thm:main}
$\{1324,3416725\}\sim 1234$.
\end{theorem}

We note that this is not the first Wilf-equivalence between a singleton pattern and a set of more than one pattern. In \cite{AMR}, the authors proved that the pattern $1342$ is Wilf-equivalent to the \emph{infinite} set of patterns $B=\{(2,2m-1,4,1,6,3,8,5,\dots,2m,2m-3)\ |\ m=2,3,4,\dots\}$. However, we prove a Wilf-equivalence between a singleton pattern and a \emph{finite} set of more than one pattern. Our proof is an extension of an idea of B\'{o}na \cite{Bona} in his proof that $|\Av_n(1324)|>|\Av_n(1234)|$ for $n\ge 7$. In other words, we shall reduce the bijection on permutations to a bijection on certain $(0,1)$-filled skew-Ferrers boards.

Given a permutation $\sigma\in S_n$, consider an $n\times n$ board $M_\sigma$ filled with 0's and 1's, so that the 1's are in cells in column $i$ (from left to right) and row $\sigma(i)$ (from bottom to top), for $1\le i\le n$, and 0's are in all other positions.

We say that $\sigma$ has a \emph{left-to-right (LR) minimum} at position $i$ if $\sigma(j)>\sigma(i)$ for all $j<i$. Likewise, we say that $\sigma$ has a \emph{right-to-left (RL) maximum} at position $i$ if $\sigma(j)<\sigma(i)$ for all $j>i$. Note that every RL-maximum is either a LR-minimum or is above and to the right of a LR-minimum. Additionally, the leftmost entry of a permutation is always a LR-minimum and the rightmost entry is always a RL-maximum.

Given a permutation matrix $M_\sigma$, remove all the boxes of $M_\sigma$ that are not both above and to the right of some LR-minimum and below and to the left of some RL-maximum, as well as the rows and columns of $M_\sigma$ containing the LR-minima and RL-maxima. Denote the resulting (0,1)-filled board by $B_\sigma$. Note that different permutations $\sigma$ may give rise to the same $B_\sigma$, e.g., $B_{23145}=B_{1234}$.

\begin{center}
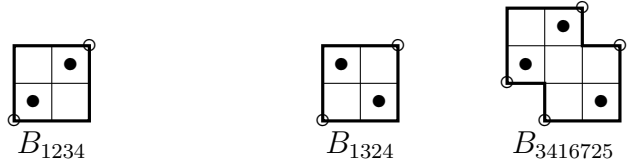
\begin{figure}[!ht]
\begin{tikzpicture}[scale=0.5]

\draw (1,0) -- (1,2);
\draw (0,1) -- (2,1);

\draw[very thick] 
     (0,0)
-- ++(0,2)
-- ++(2,0)
-- ++(0,-2)
-- cycle;

\draw (0.5,0.5)
+(0,0) node {$\bullet$} %{1}
%+(1,0) node {0}
%+(0,1) node {0}
+(1,1) node {$\bullet$} %{1}
;

\draw[color=black] (0,0)
+(0,0) node {$\circ$} %{1}
+(2,2) node {$\circ$} %{1}
;

\draw (1,0) node[below] {$B_{1234}$};

\end{tikzpicture} \qquad \qquad \qquad
\begin{tikzpicture}[scale=0.5]

\draw (1,0) -- (1,2);
\draw (0,1) -- (2,1);

\draw[very thick] 
     (0,0)
-- ++(0,2)
-- ++(2,0)
-- ++(0,-2)
-- cycle;

\draw (0.5,0.5)
%+(0,0) node {0}
+(1,0) node {$\bullet$} %{1}
+(0,1) node {$\bullet$} %{1}
%+(1,1) node {0}
;

\draw[color=black] (0,0)
+(0,0) node {$\circ$} %{1}
+(2,2) node {$\circ$} %{1}
;

\draw (1,0) node[below] {$B_{1324}$};

\end{tikzpicture} \qquad
\begin{tikzpicture}[scale=0.5]

\foreach \x in {1,2}
{
\draw (\x,0) -- (\x,3);
\draw (0,\x) -- (3,\x);
}

\draw[very thick] 
     (0,1)
-- ++(0,2)
-- ++(2,0)
-- ++(0,-1)
-- ++(1,0)
-- ++(0,-2)
-- ++(-2,0)
-- ++(0,1)
-- cycle;

\draw (0.5,0.5)
+(0,1) node {$\bullet$} %{1}
%+(0,2) node {0}
%+(1,0) node {0}
%+(1,1) node {0}
+(1,2) node {$\bullet$} %{1}
+(2,0) node {$\bullet$} %{1}
%+(2,1) node {0}
;

\draw[color=black] (0,0)
+(0,1) node {$\circ$} %{1}
+(1,0) node {$\circ$} %{1}
+(2,3) node {$\circ$} %{1}
+(3,2) node {$\circ$} %{1}
;

\draw (1.5,0) node[below] {$B_{3416725}$};

\end{tikzpicture}
\caption{Boards $B_{1234}$, $B_{1324}$, and $B_{3416725}$. Positions of 1's are denoted by the dots. The remaining positions are filled with 0's. Positions of the corresponding (removed) LR-minima and RL-maxima are denoted by circles.}
\end{figure}
\end{center}

It is easy to see that $B_\sigma$ is a (possibly empty) skew-Ferrers board with an equal number of rows and columns that contains the antidiagonal. (We will call such boards \emph{nice}.) This is because if $B_\sigma$ contains cells $(a,b)$ and cell $(c,d)$ (with rows and columns re-numbered consecutively after the cell removal as above) then $B_\sigma$ contains all cells $(k,l)$ with $a\le k\le c$ and $b\le l\le d$. To see that $B_\sigma$ contains the antidiagonal, suppose that an antidiagonal cell $c$ with coordinates $(j,r-j+1)$ is not in $B_\sigma$, where $r$ is the number of rows and columns in $B_\sigma$. Then we have two possibilities: (1) there is no cell in $B_\sigma$ either below or to the left of $c$ (or both), or (2) there is no cell in $B_\sigma$ either above or to the right of $c$ (or both). Without loss of generality, assume the former. Then every 1 in the $j$ leftmost rows of $B_\sigma$ is also contained in the $j-1$ topmost columns of $B_\sigma$ (from $r-j+2$ through $r$), which is impossible.

Define avoidance on boards as follows.

\begin{definition} \label{def:pattern-board}
A nice board $B$ \emph{contains} a nice board $B'$ if $B'$ can be obtained from $B$ by deleting an equal number of rows and columns. Otherwise, we say that $B$ \emph{avoids} $B'$.
\end{definition}

We note the following fact.

\begin{lemma} \label{lem:mb} ~\\[-10pt]
\begin{enumerate}
\item A permutation $\sigma$ avoids pattern $1234$ if and only if $B_\sigma$ avoids pattern $B_{1234}$%$={\tiny \begin{bmatrix}0&1\\1&0\end{bmatrix}}$
.
\item A permutation $\sigma$ avoids patterns $1324$ and $3416725$ if and only if $B_\sigma$ avoids patterns $B_{1324}$ %$={\tiny \begin{bmatrix}1&0\\0&1\end{bmatrix}}$ 
 and $B_{3416725}$%$={\tiny \begin{bmatrix}0&1&\phantom{0}\\1&0&0\\\phantom{0}&0&1\end{bmatrix}}$
.
\end{enumerate}
\end{lemma}

%In this notation, the empty cells in the last pattern must be outside $B_\sigma$.

\begin{proof}
Clearly, if $\sigma$ contains an occurrence of $1234$ (respectively, $1324$), then $B_\sigma$ contains $B_{1234}$ (respectively, $B_{1324}$) as a sub-board. Likewise, if $B_\sigma$ contains $B_{1234}$ (respectively, $B_{1324}$) as a sub-board, then $M_\sigma$ has a LR-minimum below and to the left, and a RL-maximum above and to the right, of that occurrence of $B_{1234}$ (respectively, $B_{1324}$); in other words, $\sigma$ contains pattern $1234$ (respectively, $1324$). Thus, $\sigma$ avoids $1234$ (respectively, $1324$) if and only if $B_\sigma$ avoids $B_{1234}$ (respectively, $B_{1324}$). 

Consider a $B_{1324}$-avoiding board $B_\sigma$. It is easy to see that if $B_\sigma$ contains $B_{3416725}$, then $M_\sigma$ has at least two LR-minima in non-consecutive rows and columns and at least two RL-maxima in non-consecutive rows and columns above and to the right of those LR-minima; these four entries correspond to the circles on the border of $B_{3416725}$. Given the position of 1's in $B_{3416725}$, this implies that $M_\sigma$ contains $M_{3416725}$, i.e., $\sigma$ contains $3416725$.

Conversely, suppose that a $1324$-avoiding permutation $\sigma$ contains an occurrence $(\sigma_{i_1},\dots,\sigma_{i_7})$ of $3416725$ with $1<i_1\le\dots\le i_7\le n$, that is, $\sigma_{i_3}<\sigma_{i_6}<\sigma_{i_1}<\sigma_{i_2}<\sigma_{i_7}<\sigma_{i_4}<\sigma_{i_5}$. We will show that in this case $\sigma$ also contains an occurrence of $3416725$ where the ``3'' and the ``1'' are LR-minima of $\sigma$ and the ``7'' and the ``5'' are RL-maxima of $\sigma$.

Define positions $j$ and $k$ as follows. If $\sigma$ has a LR-minimum at position $i_3$, let $j=i_3$. If not, let $j<i_3$ be such that $\sigma_j$ is the rightmost LR-minimum to the left of $\sigma_{i_3}$. Then $\sigma_j<\sigma_{i_3}$. If $j<i_2$, then $\sigma$ contains an occurrence of $1324$ at positions $(j,i_2,i_3,i_4)$, which is impossible. Therefore, $i_2<j<i_3$ (since $\sigma_{i_2}$ is not a LR-minimum of $\sigma$). Similarly, if $\sigma$ has a LR-minimum at position $i_1$, let $k=i_1$; otherwise, let $k<i_1$ be such that $\sigma_k$ is the smallest LR-minimum to the left of $\sigma_{i_1}$. Then $\sigma_k<\sigma_{i_1}$. If $\sigma_k<\sigma_{i_6}$, then $\sigma$ contains an occurrence of 1324 at positions $(k,i_1,i_6,i_7)$, which is impossible. Therefore, $\sigma_{i_6}<\sigma_k<\sigma_{i_1}$  (since $\sigma_{i_6}$ is not a LR-minimum of $\sigma$).

Therefore, $\sigma$ contains an occurrence of $3416725$ at positions $(k,i_2,j,i_4,i_5,i_6,i_7)$, where the ``3'' and the ``1'' (that is, $\sigma_k$ and $\sigma_j$) are LR-minima of $\sigma$. Similarly, we can find an occurrence of 3416725 in $\sigma$ where $\sigma_{i_5}$ and $\sigma_{i_7}$ (the ``7'' and the ``5'') are also (replaced with) RL-maxima of $\sigma$. But if $\sigma$ contains occurrence of $3416725$ with LR-minima as the ``3'' and the ``1'' and RL-maxima as the ``7'' and the ``5'', then $B_\sigma$ contains $B_{3416725}$.

Thus, $\sigma$ avoids both $1324$ and $3416725$ if and only if $B_\sigma$ avoids both $B_{1324}$ and $B_{3416725}$.
\end{proof}

Finally, we note that Lemma 2 of B\'{o}na \cite{Bona} is equivalent to the following statement.

\begin{lemma} \label{lem:1234}
Every nice board $B$ has a unique $B_{1234}$-avoiding (0,1)-filling with a single 1 in each row and each column.
\end{lemma}

\begin{proof}
Clearly, the 1's have to be placed exactly in all the cells on the antidiagonal of $B$, otherwise $B_{1234}$ is created.
\end{proof}

This means that to prove Theorem \ref{thm:main} we only need to prove the following lemma.

\begin{lemma} \label{lem:1324-3416725}
Every nice board $B$ has a unique $\left(B_{1324},B_{3416725}\right)$-avoiding (0,1)-filling with a single 1 in each row and each column.
\end{lemma}

\begin{proof}
Note that only the cells with 1's need to be specified, since the rest of the cells are filled with 0's.

We will insert the 1's into cells of $B$ recursively as follows. Suppose that the bottom row of $B$ has $k$ cells, and the rightmost column of $B$ has $l$ cells. Then put a 1 in the leftmost cell in the bottom row if $k\le l$ and into the top cell in the rightmost column if $k>l$. Remove the row and column containing this 1 and continue according to the same rule until no cells remain.

Proceeding inductively, we only need to prove that the position for the first 1 inserted into $B$ is unique under the above avoidance conditions. Indeed, it is easy to see that a 1 inserted as described above cannot be part of any occurrence of $B_{1324}$ or $B_{3416725}$. Thus, it is easy to see by induction that the whole filling of $B$ as above avoids both $B_{1324}$ and $B_{3416725}$.

Suppose that a 1 is inserted in a different position in the bottom row (if $k\le l$) or the rightmost column (if $k<l$). Without loss of generality, we can assume that $k\le l$ (otherwise, reflect $B$ across the antidiagonal), so that the 1 is inserted in the bottom row. Then this 1 is not in the leftmost column intersecting the bottom row of $B$, but it is in the bottom row in the $m$th column from the right, where $m<k$. Suppose the $m$th rightmost column is of height $h$. Then $h\ge l$. Consider the $(m+1)$-st column of $B$ from the right. It contains the bottom row cell immediately to the left of the 1 in the bottom row. Therefore, the 1 in the $(m+1)$-st column from the right cannot be in the bottom $h$ cells of that column (or else $B$ would contain $B_{1324}$). Thus, the $(m+1)$-st column from the right should contain more than $h$ cells, and the 1 in that column should be in row $h_1$ from the bottom for some $h_1>h$. Therefore, $B$ contains the sub-board 
%${\tiny \begin{bmatrix}1&\phantom{0}\\0&1\end{bmatrix}}=B_{14523}$
\[
\begin{tikzpicture}[scale=0.5]

\foreach \x in {1,2}
{
\draw (\x,0) -- (\x,1);
\draw (0,\x) -- (1,\x);
}

\draw[very thick] 
     (0,0)
-- ++(0,2)
-- ++(1,0)
-- ++(0,-1)
-- ++(1,0)
-- ++(0,-1)
-- cycle;

\draw (0.5,0.5)
+(0,1) node {$\bullet$} %{1}
%+(0,2) node {0}
%+(1,0) node {0}
%+(1,1) node {0}
+(1,0) node {$\bullet$} %{1}
%+(2,1) node {0}
;

\draw[color=gray] (0,0)
+(0,0) node {$\circ$} %{1}
+(1,2) node {$\circ$} %{1}
+(2,1) node {$\circ$} %{1}
;

\draw (0,1) node[left] {$B_{14523}=$};

\end{tikzpicture}
\]
in the columns $m$ and $(m+1)$ from the right and rows 1 and $h_1$ from the bottom.

Now consider the $l-1$ rows immediately above the bottom row of $B$, i.e., those that also intersect with the rightmost column. Let $l_1\in[2,l]$, and suppose the 1 in row $l_1$ from the bottom is in a column $m_1$ from the right. If $m_1\in[m+1,k]$, then $B$ contains the pattern $B_{1324}$ in the intersections of rows 1 and $l_1$ from the bottom and columns $m$ and $m_1$ from the right. If $m_1>k$, then $B$ contains the pattern $B_{3416725}$ in the intersections of rows 1, $l_1$ and $h_1$ from the bottom and columns $m$, $m+1$ and $m_1$ from the right. Therefore, to avoid both $B_{1324}$ and $B_{3416725}$, we must have $m_1\le m$.

Thus, in every one of the $l$ bottom rows, the 1's are contained in the $m$ rightmost columns. But this is a contradiction, since $m<k\le l$. This ends the proof.
\end{proof}

Such a $(B_{1324},B_{3416725})$-avoiding filling of a nice board $B$ is similar to the constructions of Simion and Schmidt~\cite{SS} and Krattenthaler~\cite{Kra} of 132-avoiding permutations starting from their LR-minima.

Finally, note that all the permutations in a class in the sense of B\'ona~\cite{Bona} are exactly those that have the same nice board (but different fillings). As a result, the single filling that avoids $(B_{1324}, B_{3416725})$ corresponds to exactly one permutation in $\Av(1324,3416725)$ corresponding to the considered board. The same applies to the single filling that avoids 1234. Thus, the sets $\Av_n(1324,3416725)$ and $\Av_n(1234)$ have the same cardinality for all $n\ge 0$. This ends the proof of Theorem~\ref{thm:main}.

\bigskip

A permutation $\sigma$ is called an \emph{involution} if $\sigma^{-1}=\sigma$. Let $I_n(\pi)$ denote the subset of involutions in $\Av_n(\pi)$.

\begin{theorem}
The set of patterns $(4231,5276143)$ is Wilf-equivalent to pattern $4321$ on involutions. Moreover, for all $n\ge 0$, we have $|I_n(4231,5276143)|=|I_n(4321)|=M_n$, the $n$th Motzkin number.
\end{theorem}

\begin{proof}
Note that if a board $B$ as in Lemma \ref{lem:1324-3416725} is symmetric about the antidiagonal then the $\left(B_{1324},B_{3416725}\right)$-avoiding (0,1)-filling of $B$ is also symmetric about the antidiagonal (since it is unique, and reflection across the antidiagonal preserves $B$).

Reversing $B$, we obtain a board $B'$ with equal numbers of rows and columns that contains the main diagonal. Thus, by Lemmas \ref{lem:1234} and \ref{lem:1324-3416725}, each such $B'$ contains a unique $B_{4321}$-avoiding filling and a unique $\left(B_{4231},B_{5276143}\right)$-avoiding filling. Moreover, if $B'$ is symmetric about the main diagonal, then both of those filling are symmetric about the main diagonal as well. Note that if $\sigma$ is an involution, then $B_{\sigma^{-1}}=B_{\sigma}$ is symmetric about the main diagonal. Thus, $|I_n(4231,5276143)|=|I_n(4321)|$, and it is known \cite{Regev} that $|I_n(4321)|=M_n$. This ends the proof.
\end{proof}

\section{A $(4,4)\sim(4,4,6)$ Wilf-Equivalence} \label{sec:44-446}

We start with a few preliminaries that will be needed in this section. A \emph{permutation class} is a set of permutations which is closed downward under the pattern containment order, i.e., $\CC$ is a class if whenever $\pi \in \CC$ and $\sigma \leq \pi$, we have $\sigma \in \CC$. Every permutation class can be described by the unique set of minimal permutations which it does not contain, called its \emph{basis}. The class of permutations avoiding a set of patterns $\Pi$ is then denoted $\Av(\Pi)$. 

A permutation $\pi$ of length $n$ is \emph{sum decomposable} if there exists $1 \leq i < n$ such that $\{\pi(j) : j \leq i\} = \{1,2,\ldots, i\}$. Otherwise, $\pi$ is said to be \emph{sum indecomposable}. Similarly, $\pi$ is said to be \emph{skew decomposable} if there exists $1 \leq i < n$ such that $\{\pi(j) : j > i\} = \{1, 2, \ldots, n-i\}$, and is otherwise \emph{skew indecomposable}. 

An \emph{interval} of a permutation is a nonempty contiguous set of indices $\{i, i+1, \ldots, j\}$ such the set of values $\{\pi(i), \pi(i+1), \ldots, \pi(j)\}$ is also contiguous. A permutation of length $n$ is said to be \emph{simple} if it does not contain any intervals other than those of lengths $1$ and $n$. For instance, the permutations $2413$ and $3142$ are the only simple permutations of length four. Given a class $\CC$, we use $\Si(\CC)$ to denote the set of simple permutations in the class $\CC$. 

Simple permutations represent the fundamental building blocks of permutations via the operation of inflation. Given a permutation $\sigma$ of length $k$ and a sequence of nonempty permutations $(\tau_i)_{i=1}^k$, the \emph{inflation} of $\sigma$ by $(\tau_i)_{i=1}^k$, written as $\sigma[\tau_1, \tau_2, \ldots, \tau_k]$, is the permutation of length $|\tau_1| + \cdots + |\tau_k|$ such that each entry $\sigma(i)$ is replaced by the interval $\tau_i$. For example,
	\[3142[123,1,21,312] = 567198423.\]
	
Note that the sum decomposable (resp., skew decomposable) permutations are exactly those which are inflations of $12$ (resp., $21$). Moreover, we write $12[\sigma,\tau]$ as $\sigma \oplus \tau$ and $21[\sigma,\tau]$ as $\sigma \ominus \tau$.
	
The following lemma allows us to derive information about a class of permutations by looking at the simple permutations in the class.
\begin{prop}[Albert and Atkinson~\cite{AA}]\label{lemma:decomposition}
Given a permutation $\pi$, there exists a unique simple permutation $\sigma$ such that $\pi=\sigma[\tau_1,\ldots,\tau_k]$.  When $\sigma \not\in \{12,21\}$, the intervals $\tau_1,\ldots,\tau_k$ are uniquely determined.  When $\sigma=12$ (resp., $\sigma=21$), the intervals are unique if we require the first of the two intervals to be sum (resp., skew) indecomposable.
\end{prop}

We can now prove the main theorem of this section. We do this by finding a structural description of the class, then using this description to set up a functional equation satisfied by the generating function for the class.

\begin{theorem} \label{thm:main2}
$\{2143,3142,246135\}\sim \{2413,3142\}$.
\end{theorem}

Define $\CC = \Av(2143,3142,246135)$. It is well-known that the set of permutations which avoid the patterns $\{2413,3142\}$, known as the \emph{separable permutations}, are counted by the large Schr\"oder numbers. Accordingly, we prove Theorem~\ref{thm:main2} by showing directly that $\CC$ is also counted by the Schr\"oder numbers, as has recently been conjectured by Egge~\cite{Egge}.

The \emph{skew-merged permutations} are the class of permutations which are the union of an increasing sequence and a decreasing sequence. Stankova~\cite{S} proved that the skew-merged permutations are exactly $\Av(2143,3412)$. Recently, Albert and Vatter~\cite{AV} enumerated the simple skew-merged permutations, noting that precisely half of them contain $2413$ and avoid $3142$ while the other half contain $3142$ and avoid $2413$. (While they were not the first to provide this enumeration, it is their technique which we adapt here.) Remarkably, the simple permutations of $\CC$ coincide exactly with the simple skew-merged permutations which contain $2413$, as we now prove.

\begin{lemma}
	$\Si(\Av(2143,3142,246135)) = \Si(\Av(2143,3142,3412))$
\end{lemma}
\begin{proof}
	Since $246135$ contains $3412$, it is trivial that
		\[\Si(\Av(2143,3142,246135)) \supseteq \Si(\Av(2143, 3142, 3412)).\]
	We will now prove the reverse inclusion. Suppose toward a contradiction that \linebreak $\sigma \in \Si(\Av(2143,3142,246135))$ and that $\sigma$ contains $3412$. Without loss of generality, we may pick the $3412$ pattern so that the $3$ is as leftmost as possible, the $4$ is as topmost as possible (i.e., has value as great as possible) for the chosen $3$, the $1$ is as bottommost as possible (i.e., has value as small as possible) for the chosen $3$ and $4$, and the 2 is as rightmost as possible for the given $3$, $4$, and $1$. 
	
	The figure below shows this situation as illustrated by a \emph{permutation diagram}, generated using Albert's \emph{PermLab}~\cite{APL} application. The dots represent permutation entries while the squares represent possible insertion locations for future entries: white squares represent valid insertion locations, dark gray squares represent insertion locations  that are forbidden because such an insertion would form a basis element, and light gray squares represent insertion locations that we have assumed are empty (for example, by assuming that the $3$ is as leftmost as possible).
	\begin{center}
		%Tikz output
		\definecolor{light-gray}{gray}{0.5}
		\definecolor{dark-gray}{gray}{0.35}
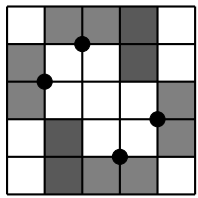
\begin{tikzpicture}[scale=.5]
		%Forbidden regions
		\filldraw[light-gray](0,2) rectangle (1,3);
		\filldraw[light-gray](0,3) rectangle (1,4);
		\filldraw[dark-gray](1,0) rectangle (2,1);
		\filldraw[dark-gray](1,1) rectangle (2,2);
		\filldraw[light-gray](1,4) rectangle (2,5);
		\filldraw[light-gray](2,0) rectangle (3,1);
		\filldraw[light-gray](2,4) rectangle (3,5);
		\filldraw[light-gray](3,0) rectangle (4,1);
		\filldraw[dark-gray](3,3) rectangle (4,4);
		\filldraw[dark-gray](3,4) rectangle (4,5);
		\filldraw[light-gray](4,1) rectangle (5,2);
		\filldraw[light-gray](4,2) rectangle (5,3);
		%Points
		\draw[black, fill=black] (1,3) circle (0.2);
		\draw[black, fill=black] (2,4) circle (0.2);
		\draw[black, fill=black] (3,1) circle (0.2);
		\draw[black, fill=black] (4,2) circle (0.2);
		%Gridlines
		\draw[thick](0,0)--(0,5);
		\draw[thick](0,0)--(5,0);
		\draw[thick](1,0)--(1,5);
		\draw[thick](0,1)--(5,1);
		\draw[thick](2,0)--(2,5);
		\draw[thick](0,2)--(5,2);
		\draw[thick](3,0)--(3,5);
		\draw[thick](0,3)--(5,3);
		\draw[thick](4,0)--(4,5);
		\draw[thick](0,4)--(5,4);
		\draw[thick](5,0)--(5,5);
		\draw[thick](0,5)--(5,5);
		\end{tikzpicture}
	\end{center}
	
As $\sigma$ is a simple permutation, the two entries in the bottom-right must be separated by an entry in either the white square directly above them, the white square directly to their left, or the white square to their far left. One can quickly see that any attempt to split them using the two adjacent white cells will eventually require a  entry located to the left of the $12$ interval in the top-left of the permutation diagram. This additional entry is shown in the following permutation diagram.

	\begin{center}
		%Tikz output
		\definecolor{light-gray}{gray}{0.5}
		\definecolor{dark-gray}{gray}{0.35}
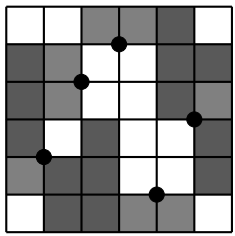
\begin{tikzpicture}[scale=.5]
		%Forbidden regions
		\filldraw[light-gray](0,1) rectangle (1,2);
		\filldraw[dark-gray](0,2) rectangle (1,3);
		\filldraw[dark-gray](0,3) rectangle (1,4);
		\filldraw[dark-gray](0,4) rectangle (1,5);
		\filldraw[dark-gray](1,0) rectangle (2,1);
		\filldraw[dark-gray](1,1) rectangle (2,2);
		\filldraw[light-gray](1,3) rectangle (2,4);
		\filldraw[light-gray](1,4) rectangle (2,5);
		\filldraw[dark-gray](2,0) rectangle (3,1);
		\filldraw[dark-gray](2,1) rectangle (3,2);
		\filldraw[dark-gray](2,2) rectangle (3,3);
		\filldraw[light-gray](2,5) rectangle (3,6);
		\filldraw[light-gray](3,0) rectangle (4,1);
		\filldraw[light-gray](3,5) rectangle (4,6);
		\filldraw[light-gray](4,0) rectangle (5,1);
		\filldraw[dark-gray](4,3) rectangle (5,4);
		\filldraw[dark-gray](4,4) rectangle (5,5);
		\filldraw[dark-gray](4,5) rectangle (5,6);
		\filldraw[dark-gray](5,1) rectangle (6,2);
		\filldraw[dark-gray](5,2) rectangle (6,3);
		\filldraw[light-gray](5,3) rectangle (6,4);
		\filldraw[dark-gray](5,4) rectangle (6,5);
		%Points
		\draw[black, fill=black] (1,2) circle (0.2);
		\draw[black, fill=black] (2,4) circle (0.2);
		\draw[black, fill=black] (3,5) circle (0.2);
		\draw[black, fill=black] (4,1) circle (0.2);
		\draw[black, fill=black] (5,3) circle (0.2);
		%Gridlines
		\draw[thick](0,0)--(0,6);
		\draw[thick](0,0)--(6,0);
		\draw[thick](1,0)--(1,6);
		\draw[thick](0,1)--(6,1);
		\draw[thick](2,0)--(2,6);
		\draw[thick](0,2)--(6,2);
		\draw[thick](3,0)--(3,6);
		\draw[thick](0,3)--(6,3);
		\draw[thick](4,0)--(4,6);
		\draw[thick](0,4)--(6,4);
		\draw[thick](5,0)--(5,6);
		\draw[thick](0,5)--(6,5);
		\draw[thick](6,0)--(6,6);
		\draw[thick](0,6)--(6,6);
		\end{tikzpicture}
	\end{center}
	
	Such an entry now limits the ways that we may split the remaining $12$ interval. Entries placed in the white squares directly below or to the right of the $12$ interval will only create larger intervals that cannot be separated. This implies that $\sigma$ will always contain a nontrivial interval, which contradicts the assumption that $\sigma$ is simple.
\end{proof}

In describing a permutation class, it is often helpful to consider four different types of permutations: the single permutation of length $1$, the sum decomposable permutations, the skew decomposable permutations, and the permutations that are inflations of simple permutations of length at least $4$.

The previous lemma allows us to count the permutations in $\CC$ which are inflations of simple permutations of length at least 4. Consider the simple permutation of $\CC$ shown below. The black entries are part of the increasing subsequence, the gray entries are part of the decreasing subsequence, and the circled entry is the \emph{central point} which may or may not exist in a simple skew-merged permutation.

	
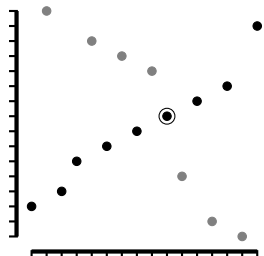
\begin{center}
		\begin{tikzpicture}[scale=.2,baseline=(current bounding box.center)]
			\draw[ultra thick] (1,0) -- (16,0);
			\draw[ultra thick] (0,1) -- (0,16);
			\foreach \x in {1,...,16} {
				\draw[thick] (\x,.09)--(\x,-.5);
				\draw[thick] (.09,\x)--(-.5,\x);
			}
			\draw[fill=black] (1,3) circle (8pt);
			\draw[gray, fill=gray] (2,16) circle (8pt);
			\draw[fill=black] (3,4) circle (8pt);
			\draw[fill=black] (4,6) circle (8pt);
			\draw[gray, fill=gray] (5,14) circle (8pt);
			\draw[fill=black] (6,7) circle (8pt);
			\draw[gray, fill=gray] (7,13) circle (8pt);
			\draw[fill=black] (8,8) circle (8pt);
			\draw[gray, fill=gray] (9,12) circle (8pt);
			\draw[fill=black] (10,9) circle (8pt);
			\draw (10,9) circle (15pt);
			\draw[gray, fill=gray] (11,5) circle (8pt);
			\draw[fill=black] (12,10) circle (8pt);
			\draw[gray, fill=gray] (13,2) circle (8pt);
			\draw[fill=black] (14,11) circle (8pt);
			\draw[gray, fill=gray] (15,1) circle (8pt);
			\draw[fill=black] (16,15) circle (8pt);
		\end{tikzpicture}
	\end{center}
	
Entries in the increasing subsequence can only be inflated by increasing permutations, since a decrease would form a $2143$ pattern. Entries in the decreasing subsequence and the central point can all be inflated by any permutation in the class.

The generating function for the simple skew-merged permutations given by Albert and Vatter does not distinguish between the entries in the increasing and decreasing subsequences. Fortunately, we can easily adapt their method to count each type of entry separately, in the same way that B\'ona, Homberger, Pantone, and Vatter~\cite{BHPV} adapted the method Albert and Vatter used to count $321$-avoiding simple permutations, distinguishing between entries in each of the two increasing subsequences. Performing this calculation, which we omit due to its similarity to the aforementioned references, one finds that the simple $3142$-avoiding skew-merged permutations are counted by the generating function
	\[s(u,v) = \frac{2u^2v^2(1+v)}{1 - 2uv(u+2) - uv^2(u+2) + (1-uv)\sqrt{1-2uv(2u+3) - uv^2(3u+4)}},\]
where $u$ counts entries in the increasing subsequence (omitting a central point if one exists) and $v$ counts entries in the decreasing subsequence (including a central point if one exists). Let $f$ denote the generating function for the class $\CC$. By our earlier observations about the allowed inflations, we conclude that the permutations in $\CC$ which are inflations of the simple permutations of length at least $4$ are counted by 
	\[s\left(\frac{x}{1-x},f\right).\]
	
	It only remains to count the sum and skew decomposable permutations of $\CC$. Let $f_\oplus$ and $f_\ominus$ denote the sum decomposable and skew decomposable permutations in $\CC$, respectively. As in the case studied by Albert and Vatter~\cite{AV}, the sum decomposable permutations in this class all have the form $1 \oplus \pi$ or $\pi \oplus 1$. Counting these, and making sure not to double-count those permutations of the form $1 \oplus \pi \oplus 1$, we see that
		\[f_\oplus = 2xf - x^2(f+1).\]
	
	Moreover, the class $\CC$ is closed under the operation of skew sum (which is readily derived from the fact that $\CC$ has no skew decomposable basis elements) from which it follows in the usual manner that
		\[f_\ominus = \frac{f^2}{1+f}.\]
	
	Combining these results, $f$ satisfies the functional equation
		\begin{equation} \label{eq:f-fe}
		\begin{split}
			f &= x + f_\oplus + f_\ominus + s\left(\frac{x}{1-x},f\right)\\
			&= x + 2xf - x^2(f+1) + \frac{f^2}{1+f} + s\left(\frac{x}{1-x},f\right).
		\end{split}
		\end{equation}

	It can then be verified by a computer algebra system that one solution to \eqref{eq:f-fe} is
		\[f = \frac{1-x-\sqrt{1-6x+x^2}}{2},\]
	which is the generating function for the large Schr\"oder numbers. In fact, $f$ is the unique formal power series solution to \eqref{eq:f-fe}. To see this, note that substituting $x=0$ into \eqref{eq:f-fe} implies that any solution $f$ must have constant term $0$. It follows then that the coefficient of $x^n$ in the power series of $f$ at $x=0$ is uniquely determined by the coefficients of $x^1, x^2, \ldots, x^{n-1}$. \label{fact:unique}
	
	Therefore, the permutation class $\Av(2143,3142,246135)$ is counted by the large Schr\"oder numbers, and hence is Wilf-equivalent to $\Av(2413,3142)$.

\section{Some Conjectures} \label{sec:conj}

There appear to be more Wilf-equivalences between a singleton pattern and a set of more than one pattern. In fact, we conjecture the following Wilf-equivalence.

\begin{conjecture} \label{conj:2143-246135}
$\{2143,246135\}\sim 2413$.
\end{conjecture}

Since $|\Av_5(\pi)| = 103$ for all $\pi$ with $|\pi|=4$, there can be no $(4)\sim(4,5)$ Wilf-equivalence; for if $|\alpha| = 4$ and $|\beta| = 5$ with $\alpha \not\leq \beta$, then it must follow that $|\Av_5(\alpha,\beta)| = 102$. Computation has verified that, up to symmetries, there are no unbalanced Wilf-equivalences of the form $(4)\sim(4,6)$ or $(4)\sim(4,7)$ other than the $(4)\sim(4,7)$ Wilf-equivalence proved in this paper and the $(4)\sim(4,6)$ Wilf-equivalence conjectured above. It follows that there are no other $(4)\sim(4,k)$ Wilf-equivalences for any $k$.

It would be interesting to give a bijection in the proof of Conjecture \ref{conj:2143-246135} that would also preserve 3142-avoidance. This would yield another proof that
\[
\{2143,3142,246135\}\sim \{2413,3142\}.
\]

In fact, Egge \cite{Egge} has conjectured that $\{2143,3142,\pi\} \sim \{2413,3142\}$ for any permutation $\pi$ among
\[
246135,\ 254613,\ 263514,\ 362415,\ 461325,\ 524361,\ 546132,\ 614352.
\]

Additionally, the authors discovered after submission that Jel\'{i}nek~\cite{Jelinek}, in his Doctoral Thesis, proved that $1234 \sim \{1324,3416725\}$ using largely similar methods. However, this result was never published. Moreover, Jel\'{i}nek~\cite{Jelinek-private} has indicated that he has recently proved the following generalization: 
	\[\{\sigma \oplus 12 \oplus \tau\} \sim \left(\{\sigma \oplus 21 \oplus \tau\} \cup B\right),\]
where $B$ is a (necessarily finite) set of permutations of length at most $2|\sigma| + 2|\tau| + 3$.

\section*{Acknowledgement}

The authors would like to thank Mikl\'{o}s B\'{o}na for his close reading of the paper, helpful comments in clarifying the exposition, and pointing out the connection of this paper to~\cite{AMR}. The authors are also grateful to the anonymous referees for their useful suggestions in improving the clarity of the paper.

\end{document}